\documentclass[reqno,10pt]{amsart}
\usepackage{amsthm}
\usepackage{amsmath}
\usepackage{amssymb}
\usepackage{amsfonts}
\usepackage{mathrsfs}
\usepackage{todonotes}
\usepackage{verbatim}
\usepackage[bookmarksnumbered,colorlinks,urlcolor=blue]{hyperref}
\usepackage{graphicx}
\usepackage{graphics}
\usepackage{float}
\usepackage{enumerate}
\usepackage{ulem}
\def\bb#1\eb{\textcolor{blue}
{#1}} %
\def\br#1\er{\textcolor{red}
{#1}} %
\DeclareMathOperator{\comp}{length}

\definecolor{creme}{RGB}{255,253,150}
\definecolor{verdeescuro}{RGB}{47,79,47}

\newcommand{\R}{\mathds R}

\newcommand{\qcd}{\begin{flushright} $\Box$ \end{flushright}}

   \def\br#1\er{\textcolor{red}{#1}} %
      \def\bb#1\eb{\textcolor{blue}{#1}} %

\title[]{Locally extremal timelike geodesic loops on Lorentzian manifolds}

\author[I.P. Costa e Silva]{Ivan P. Costa e Silva}
\address{Department of Mathematics, Universidade Federal de Santa Catarina,
\hfill\break\indent 88.040-900
Florianópolis-SC, Brazil.}
\email{pontual.ivan@ufsc.br}

\author[J.L. Flores]{Jos\'e L. Flores}
\address{Departamento de \'Algebra, Geometr\'{\i}a y Topolog\'{\i}a,  Universidad de M\'alaga
\hfill\break\indent
Facultad de Ciencias, Campus Universitario de Teatinos,
\hfill\break\indent 29071 M\'alaga, Spain}
\email{floresj@uma.es}

\author[K.P.R. Honorato]{Kledilson P. R. Honorato}
\address{Department of Mathematics, Universidade Federal de Santa Catarina,
	\hfill\break\indent 88.040-900
	Florianópolis-SC, Brazil.}
\email{kledilson.honorato@posgrad.ufsc.br}

\begin{document}
\newtheorem{thm}{Theorem}[section]
\newtheorem{prop}[thm]{Proposition}
\newtheorem{lemma}[thm]{Lemma}
\newtheorem{cor}[thm]{Corollary}
\theoremstyle{definition}
\newtheorem{defi}[thm]{Definition}
\newtheorem{notation}[thm]{Notation}
\newtheorem{exe}[thm]{Example}
\newtheorem{conj}[thm]{Conjecture}
\newtheorem{prob}[thm]{Problem}
\newtheorem{rem}[thm]{Remark}
\newtheorem{conv}[thm]{Convention}
\newtheorem{crit}[thm]{Criterion}
\newtheorem{claim}[thm]{Claim}

\newcommand{\ben}{\begin{enumerate}}
\newcommand{\een}{\end{enumerate}}

\newcommand{\bit}{\begin{itemize}}
\newcommand{\eit}{\end{itemize}}

\begin{abstract}
Conditions for the existence of closed geodesics is a classic, much-studied subject in Riemannian geometry, with many beautiful results and powerful techniques. However, many of the techniques that work so well in that context are far less effective in Lorentzian geometry. In revisiting this problem here, we introduce the notion of timelike geodesic homotopy, a restriction to geodesics of the more standard timelike homotopy (also known as $t$-homotopy) of timelike loops on Lorentzian manifolds. This tool is combined with a local shortening/stretching of length argument to provide a number of new results on the existence of closed timelike geodesics on compact Lorentz manifolds. 
%
\end{abstract}

\maketitle
\section{Introduction} 

A well-known result in Riemannian geometry establishes the existence of at least one periodic geodesic in any compact Riemannian manifold. This result has no general analog in Lorentzian geometry;   
nevertheless, there are multiple partial results in this direction. We focus here only on the (non-)existence of closed {\it causal}, and especially {\it timelike} geodesics. In \cite{tipler} Tipler proved the existence of one closed timelike geodesic in any compact Lorentzian manifold admitting a regular covering by a globally hyperbolic spacetime with compact Cauchy hypersurfaces. Guediri \cite{guediri02,guediri03,guediri07} and S\'anchez \cite{sanchez} obtained the same conclusion in the case of a compact Lorentzian manifold which is again regularly covered by a globally hyperbolic spacetime, but whose Cauchy hypersurfaces are not necessarily compact; instead, the group of deck transformations of the cover is assumed to possess certain extra properties. In \cite{galloway}, Galloway proved the existence of a longest periodic timelike curve - which is necessarily a closed timelike geodesic - in each {\it stable} free timelike homotopy class. In \cite{galloway86}, the same author proved the existence of a causal closed geodesic in any compact two-dimensional Lorentzian manifold. More recently, Guediri has proved that certain compact flat spacetimes contain a causal closed geodesic \cite{guediri02}, and that such spacetimes contain a closed timelike geodesic if and only if the fundamental group of the underlying manifold contains a nontrivial timelike translation \cite{guediri03}. Nonexistence results for periodic causal geodesics are also available, see \cite{galloway86,guediri03B,guediri07B}.


The general problem of existence of periodic geodesics also has interesting partial solutions in the case of compact Lorentz manifolds with some added symmetry. For instance, by using variational methods in static spacetimes it has been established the existence of a closed geodesic in each free homotopy class corresponding to an element of the fundamental group having finite conjugacy class \cite{CMP}. This result has been generalized to any free homotopy class containing a periodic timelike curve in \cite{sanchez}, leading to the conclusion that any compact static spacetime admits a closed timelike geodesic. In \cite{FJP} the authors obtain an existence result for periodic geodesics in Lorentzian manifolds endowed with a Killing vector field which is timelike at least at one point, by using the strong relation between Killing vector fields and geodesics, together with a compactness criterion for subgroups of the isometry group.


The difficulty in establishing general existence results for (say) closed timelike geodesics stems from a number of technical peculiarities of Lorentzian geometry as compared to the Riemannian case, which make it problematic to apply variational techniques which are both natural and fruitful in the latter context, but not in the former. These differences manifest, for example, in the relatively unwieldy structure of the space of closed causal curves (endowed with suitable topologies) on a compact Lorentzian manifold. The length functional defined thereon has in general neither a natural lower bound, nor an upper bound. The tremendously useful Palais-Smale condition which is satisfied by the Riemannian length functional fails here as well. In addressing these issues, our philosophy here is that some of these difficulties can be ameliorated by working directly on the space of (timelike) {\it geodesics}, instead of that of (timelike) curves.

The purpose of this paper is twofold. First, we introduce and discuss the notion of {\it timelike geodesic homotopy} (Definition \ref{tghomotopy}), which is a restriction to geodesics of the more standard timelike homotopy of timelike loops; the new notion defines an equivalence relation on the set of timelike geodesic loops. We then proceed to show how timelike geodesic homotopies can be used, coupled with a stretching/shortening of length argument embodied in the Theorem \ref{stretch} - to which we refer, somewhat picturesquely, as the {\it Hill-Climb Lemma} - to obtain a number of new results on the existence of closed timelike geodesics in compact Lorentz manifolds (see especially Proposition \ref{newmainthm'} and Theorem \ref{newmainthm}).

 Since each timelike geodesic homotopy class consists just of timelike geodesic loops, the degrees of freedom involved in the search of closed timelike geodesics in such a class is considerably reduced, thereby simplifying certain variational arguments. Indeed, since the geodesic character of the loop is already guaranteed at the outset, the problem then reduces to that of finding some loop in the class whose initial and final velocities are the same. In order to solve the latter problem, we have developed a Lorentzian adaptation of the length-shortening/stretching argument introduced in the context of Riemannian geometry in \cite{F} by one of the authors (JLF). That simplification, however, comes at the cost of assuming the absence of certain type of conjugate points along the loops. (While this remains an important added assumption, it may well be an artifact of our particular approach.)

Another difficulty in the approach adopted here is to ensure the existence of even one timelike geodesic loop, in order to meaningfully define its timelike geodesic homotopy class. Solving this problem in full generality is a nontrivial problem in Lorentzian geometry; fortunately, it is one which can be bypassed in a number of interesting particular cases. Our examples here include Theorem \ref{thmclave3} and Proposition \ref{coolcor} below (see also \cite{costaesilva&flores&honorato21}).

The rest of the paper is organized as follows. In Section \ref{sec1} we introduce the basic definitions and conventions we will need, including the central definition of timelike geodesic homotopy (Definition \ref{tghomotopy}), as well as some technical preliminaries useful for the rest of the paper. In Section \ref{sec1.1} we establish and prove the key Hill-Climb Lemma (Theorem \ref{stretch}). As indicated above, this is a variational result in which we show that timelike geodesic loops without certain types of conjugate points can be either stretched or shortened in a given geodesic homotopy class {\it unless} they are already closed timelike geodesics. This in turn implies, in particular, that if such a timelike geodesic loop has either maximal or minimal length in a given class (we call such loops {\it locally extremal} timelike geodesic loops), then they are closed timelike geodesics. We apply these ideas in Section \ref{sec2} to obtain our main existence result for these curves in the compact Lorentzian case (Proposition \ref{newmainthm'} and Theorem \ref{newmainthm}). In the final Section \ref{sec3} we also provide alternative geometric conditions which provide bounds on the lengths of loops as required in our approach.

\section{Basic definitions \& technical preliminaries}\label{sec1}
Throughout this note, $(M,g)$ will denote a smooth Lorentzian manifold of dimension $n\geq 2$, not necessarily compact or geodesically complete in any sense. 

The exponential map on $(M,g)$ is denoted by $\exp :\mathcal{D}\subset TM\rightarrow M$, where $\mathcal{D}$ is understood to be its maximal domain, and as usual for any $p\in M$ with tangent space $T_pM$, we write $\mathcal{D}_p := \mathcal{D}\cap (T_pM)$ and $\exp _p:= \exp |_{\mathcal{D}_p}$. The {\it norm} of a vector $v \in TM$ is denoted by $|v|=|v|_g= \sqrt{|g(v,v)|}$, and we omit the subscript `$g$' when there is no risk of confusion.

We assume in what follows that the reader is familiar with at least the basic aspects of Lorentzian geometry as found in the core references \cite{bookbeem96,bookoneill83}. Yet, in order to avoid ambiguities and establish notation we briefly recall some elementary geometric notions pertaining to Lorentzian manifolds, referring to the cited textbooks for further details. 

A nonzero tangent vector $v\in T_pM$ is {\it timelike} [resp. {\it null} (aka {\it lightlike}), {\it spacelike}] if $g(v,v)<0$ [resp. $g(v,v)=0$, $g(v,v)>0$]. (In our convention, the zero vector is spacelike.) The vector $v$ is {\it causal} (or {\it nonspacelike}) if it is either timelike or null. (These appellations are collectively referred to as the {\it causal character} of tangent vectors.) The set $\mathcal{T}_p \subset T_pM$ of timelike tangent vectors in $T_pM$ has two connected components called {\it timecones} (at $p$), denoted by $\mathcal{T}^{\pm}_p$. (These are actual open convex cones in $T_pM$.) The boundary of $\mathcal{T}_p$ in $T_pM$ consists of the set $\Lambda_p$ of null vectors at $p$ together with the zero vector $0_p$. $\Lambda_p$ also has two connected components $\Lambda ^{\pm}_p$, the {\it lightcones} at $p$, which we label so that $\Lambda^{\pm}_p\cup \mathcal{T}^{\pm}_p \cup \{0_p\}$ are the respective closures of $\mathcal{T}^{\pm}_p$. Finally, the sets $\mathcal{C}^{\pm}_p := \Lambda^{\pm}_p\cup \mathcal{T}^{\pm}_p$ are called the {\it causal cones} at $p$, and together they of course comprise the set of all causal vectors at $p$, which we denote by $\mathcal{C}_p:= \mathcal{C}^{+}_p\cup \mathcal{C}^{-}_p$. 

A {\it time orientation} on $(M,g)$ is a continuous choice $ p\in M \mapsto \mathcal{C}^{+}_p \subset T_pM$ of a causal cone/timecone at each point (called then the {\it future} causal cone/timecone, whereas the other one is the {\it past} causal cone/timecone). It is well-known that any Lorentzian manifold either admits such a time-orientation or else it has a double covering manifold which does. If a time orientation can and has been chosen, $(M,g)$ is said to be {\it time-oriented}. A {\it spacetime} is then a time-oriented, connected Lorentzian manifold. 

The causal character can be extended to curves and vector fields as follows. A piecewise smooth curve $\alpha :I \subset \mathbb{R} \rightarrow M$ is {\it timelike} [resp. {\it causal}, {\it null}, spacelike] if each tangent vector $\dot{\alpha}(t) $ along $\alpha$ is of the respective causal character, and on causal curves we require in addition that the two tangent vectors at an eventual break of $\alpha$ be both in the same causal cone. (Timelike/causal/null/spacelike vector fields are defined in an obvious similar fashion.) By a {\it closed} timelike/null/causal curve we mean a piecewise curve segment $\alpha:[a,b] \rightarrow M$ of the corresponding causal type with $\alpha(a)=\alpha(b)$. We often refer to closed curves in this sense as {\it loops}.

The (Lorentzian) {\it length} of a causal curve segment $\alpha:[a,b] \rightarrow M$ is given by
$${\rm length}(\alpha) = {\rm length}_g(\alpha) = \int_a^b |\dot{\alpha}(t)|_g \, dt,$$
and again we omit the subscript `$g$' where there is no risk of confusion. (Observe that smooth reparametrizations of causal curves are still causal curves with the same length as the original curve, and that null curves have zero length in spite of being non-constant.) 


Although a general piecewise smooth curve need not have a definite causal character, a (non-constant) {\it geodesic} always does: causal geodesics are either timelike or null, exclusively. In particular, a causal geodesic segment $\gamma: [a,b] \rightarrow M$ is timelike if and only its length is positive:
$${\rm length}(\gamma) = |\dot{\gamma}(a)| \, (b-a) >0.$$



 Now, a (non-constant) geodesic segment $\gamma: [a,b] \rightarrow M$ on $(M,g)$ is a {\it geodesic loop} if $\gamma(a)=\gamma(b)$, and such a geodesic loop is a {\it closed} geodesic
 if $\dot{\gamma}(a) = c\cdot \dot{\gamma}(b)$ for some number $c>0$. (In the latter case, if $\gamma$ is not null then one necessarily has $c\equiv 1$, so the geodesic is periodic \cite[Thm. 7.13, p. 192]{bookoneill83}.) 
 
 As discussed in the Introduction, in general one has neither a positive lower bound nor a finite upper bound on Lorentzian lengths of timelike loops. Hence, in order to find closed timelike geodesics via variational methods one must restrict the problem accordingly. As first discussed by Galloway \cite{galloway}, it is natural to impose that lengths of curves in the {\it timelike homotopy class} of a timelike loop admit a uniform upper bound. 
 
 Recall that two timelike loops $\alpha_1, \alpha_2$ are said to be {\it (freely) timelike homotopic} (or {\it $t$-homotopic} for short) if there exists a free homotopy between $\alpha_1$ and $\alpha_2$ whose longitudinal curves are all timelike loops. Here, however, we are interested in timelike {\it geodesic} loops; in that context, we will need a more stringent notion.  
 
 \begin{defi}[TG-homotopy]\label{tghomotopy}
 Two timelike geodesic loops $\gamma _i: [a,b] \rightarrow M$ ($i=1,2$) on $(M,g)$ are (freely) {\it $TG$-homotopic} if there exists a continuous map $\sigma: [0,1]\times [a,b]\rightarrow M$ such that 
 \begin{itemize}
     \item[i)] $\sigma(0,t) = \gamma_1(t)$ and $\sigma(1,t) = \gamma_2(t)$, $\forall t \in [a,b]$;
     \item[ii)] for each $s\in [0,1]$ the longitudinal curve $\sigma_s: t \in [a,b] \mapsto \sigma(s,t) \in M$ is a timelike geodesic loop. 
 \end{itemize}
 \end{defi}
In other words, two timelike geodesic loops are $TG$-homotopic if they admit a free homotopy via timelike {\it geodesic} loops. This clearly defines an equivalence relation on the set of timelike geodesic loops, and the equivalence classes are naturally dubbed {\it $TG$-homotopy classes}. In particular, $TG$-homotopic timelike geodesic loops are $t$-homotopic, so the $t$-homotopy class of a timelike geodesic loop $\gamma$ properly contains the $TG$-homotopy class of $\gamma$. 

This discussion implies that {\it any} upper bound on the lengths of timelike loops in a given $t$-homotopy class of a timelike geodesic loop is {\it ipso facto} an upper bound on the lengths of timelike geodesic loops in the $TG$-homotopy class of said loop. To be more precise, let $\gamma$ be a timelike geodesic loop, and denote by $\mathfrak{T}(\gamma)$ its $t$-homotopy class, and by $\mathfrak{TG}(\gamma)$ its $TG$-homotopy class. Then 
$$\mathfrak{TG}(\gamma)\subset \mathfrak{T}(\gamma),$$
and therefore,   
\begin{equation}\label{upperbound}
    L_{\mathfrak{TG}}(\gamma) := \sup _{\alpha \in \mathfrak{TG}(\gamma)} {\rm length}(\alpha) \leq \sup _{\alpha \in \mathfrak{T}(\gamma)} {\rm length}(\alpha) =: L_{\mathfrak{T}}(\gamma). 
\end{equation}
However, the use of timelike geodesic loops rather than just timelike loops has an added bonus: there may well be a {\it positive lower bound} on the lengths of timelike geodesic loops in the $TG$-homotopy class $\mathfrak{TG}(\gamma)$ although the lower bound of the lengths in $\mathfrak{T}(\gamma)$ is zero, i.e., we may have 
\begin{equation}\label{lowerbound}
    l_{\mathfrak{TG}}(\gamma) := \inf _{\alpha \in \mathfrak{TG}(\gamma)} {\rm length}(\alpha) >0 = \inf _{\alpha \in \mathfrak{T}(\gamma)} {\rm length}(\alpha) =: l_{\mathfrak{T}}(\gamma). 
\end{equation}

The following simple examples illustrate these considerations.
\begin{exe}\label{mink}
In two-dimensional Minkowski spacetime $(\hat{M},\hat{g})= (\mathbb{R}^2, -dt^2 + dx^2)$ (with standard time orientation such that $\partial/\partial t$ is future-directed) we introduce the isometric identification along the (timelike) $t$-axis:
$$(t,x) \sim (t+1,x).$$
The resulting quotient is of course a Lorentz manifold $(M,g)$, and topologically a cylinder. Consider the timelike geodesic loop $\gamma$ on $(M,g)$ which is the image of the timelike geodesic segment $\hat{\gamma}:t \in [0,1]\mapsto (t,0)\in \hat{M}$ by the canonical projection $\pi:\hat{M} \rightarrow M$. It is evident that $\mathfrak{TG}(\gamma)$ consists only of timelike geodesic loops $\eta$ with ${\rm length }(\eta)=1$. Therefore 
$$l_{\mathfrak{TG}}(\gamma)=L_{\mathfrak{TG}}(\gamma)=1.$$
Meanwhile, when we consider general timelike loops one still has $L_{\mathfrak{T}}(\gamma)=1$, but now $l_{\mathfrak{T}}(\gamma)=0$, since tangent vectors of these more general curves can come arbitrarily close to lightlike directions. 
\end{exe}
\begin{exe}\label{antidesitter}
Consider, in $\mathbb{R}^3$, the $2$-dimensional one-sheeted hyperboloid $M$ given by
$$ x^2+ y^2-z^2 =1,$$
and the symmetric bilinear form $\langle \, . \, , \, . \, \rangle$ of index 2 given by 
$$\langle (v_1,v_2,v_3),(w_1,w_2,w_3)\rangle = -v_1w_1-v_2w_2 + v_3w_3.$$
Upon identifying $T_p\mathbb{R}^3 \simeq \mathbb{R}^3$ for each $p \in M$ this symmetric form induces a Lorentzian metric $g$, so the hyperboloid $M$ becomes a Lorentz manifold $(M,g)$ whose universal covering is the $2$-dimensional {\it anti-de Sitter} spacetime \cite[p. 131]{bookhawking&ellis73}, which unlike the previous example is not globally hyperbolic. All geodesics thereon are reparametrizations of the intersections of planes through the origin with $M$ \cite[Ch. 4]{bookoneill83}). In particular, the closed timelike geodesics through a point $p \in M$ (observe that $\langle p,p\rangle =-1$) are described as follows. Fix any vector $e_1 \in \mathbb{R}^3$ such that $\langle e_1,e_1\rangle =-1$ and $\langle p,e_1\rangle =0$ and consider the plane $\Pi = span \, \{p,e_1\}$. The closed curve $\alpha: t \in [0,2\pi] \rightarrow \mathbb{R}^3$ given by 
$$\alpha(t) = (\cos \, t) \cdot p + (\sin \, t) \cdot e_1 $$
clearly parametrizes the intersection $\Pi \cap M$ and is thus a timelike geodesic with length $2\pi$. Since any timelike geodesic loop in the $t$-homotopy class of $\alpha$ is of this form, they will all have the same length $2\pi$. Then again
$$l_{\mathfrak{TG}}(\gamma)=L_{\mathfrak{TG}}(\gamma)=2\pi.$$

\end{exe}

One tricky aspect in dealing with timelike geodesic loops is their very existence; thus a more immediate concern is to obtain sufficient conditions for the existence of {\it at least one} timelike geodesic loop in $(M,g)$, let alone a whole $TG$-homotopy class. First, one may consider the following simple criterion. 
\begin{prop}\label{easypeasy}
Assume there exists a Lorentzian covering\footnote{Recall that if $(\hat{M},\hat{g})$ is a Lorentzian manifold, a map $\phi:\hat{M}\rightarrow M$ is a {\it Lorentzian covering map} if it is a smooth covering map for which $\hat{g} = \phi^{\ast}g$; in particular it is a local isometry.} $\phi: (\tilde{M},\tilde{g}) \rightarrow (M,g)$ where $(\tilde{M},\tilde{g})$ is a globally hyperbolic spacetime. In this case, if $(M,g)$ admits a timelike loop, then it also admits a timelike geodesic loop.
\end{prop}
\noindent {\it Proof.} Let $\alpha: [0,1] \rightarrow M$ be a timelike curve with $p:=\alpha(0)=\alpha(1)$, and pick any $\tilde{p} \in \tilde{M}$ such that $\phi(\tilde{p})=p$. Consider the lift $\tilde{\alpha}:[0,1] \rightarrow \tilde{M}$
 of $\alpha$ through $\phi$ starting at $\tilde{p}$. Since $\phi$ is a local isometry, $\tilde{\alpha}$ is a timelike curve. From standard causal properties of globally hyperbolic spacetimes $\tilde{\alpha}(0)\neq \tilde{\alpha}(1)$, and there exists a maximal - and thus necessarily timelike - causal geodesic $\tilde{\gamma}:[0,1] \rightarrow \tilde{M}$ with $\tilde{\gamma}(0)=\tilde{\alpha}(0)$ and $\tilde{\gamma}(1)= \tilde{\alpha}(1)$. Thus, $\gamma := \phi \circ \tilde{\gamma}$ is a timelike geodesic loop in $(M,g)$ at $p$. 
 \qcd

The previous proposition becomes particularly interesting when $M$ is compact; then it is well-known that a timelike loop always exists (cf., e.g., \cite[Lemma 14.10, p. 407]{bookoneill83}), in which case the situation contemplated in Prop. \ref{easypeasy} ensures that $(M,g)$ will also contain some timelike geodesic loop. In this particular context - $(M,g)$ compact and with a globally hyperbolic Lorentzian covering - the following (far from exhaustive) list of conditions are known to guarantee finiteness of $L_{\mathfrak{T}}(\gamma)$ in (\ref{upperbound}), and hence of $L_{\mathfrak{TG}}(\gamma)$, for a given timelike geodesic loop $\gamma$. (These have been already mentioned in the Introduction, and we repeat them here merely to bring the issue into sharper focus.) 

\begin{itemize}
    \item[1)] If $\mathfrak{T}(\gamma)$ is a {\it stable} $t$-homotopy class in the sense of Ref. \cite{galloway}, then $L_{\mathfrak{TG}}(\gamma)(\leq L_{\mathfrak{T}}(\gamma))<\infty$. This is a somewhat technical condition, but in that reference, Galloway gives sufficient, geometrically motivated conditions to ensure it occurs. Of particular interest vis-à-vis Prop. \ref{easypeasy} above is the demand - first considered by Tipler in \cite{tipler} - that $(M,g)$ is compact and admits a (regular) covering by a globally hyperbolic spacetime $(\tilde{M},\tilde{g})$ with compact Cauchy hypersurfaces. In this case, $\mathfrak{T}(\alpha)$ is actually stable for {\it any} timelike loop $\alpha$ in $(M,g)$, and in particular for timelike geodesic loops. 
    \item[2)] Guediri \cite{guediri02,guediri03,guediri07} extensively and specifically studied the context when $(M,g)$ is compact and admits a regular covering by a globally hyperbolic spacetime $(\tilde{M},\tilde{g})$, but not necessarily with compact Cauchy hypersurfaces. In this ambient, he displays concrete examples of such a situation in which $(M,g)$ does not admit any closed timelike geodesic, and in particular he shows (cf. \cite[Thm. 3.1]{guediri07}) that for any timelike loop $\alpha$ in $(M,g)$, $L_{\mathfrak{T}}(\alpha)<\infty$ if and only if $\mathfrak{T}(\alpha)$ contains a largest-length representative. This will occur, for instance, if the regular covering admits an abelian group of deck transformations \cite{guediri02}. 
\end{itemize}

We now turn to alternative conditions meant to ensure the existence of timelike geodesic loops in $(M,g)$ when the latter is not in principle covered by some globally hyperbolic spacetime. Our first observation is that this problem can be linked to the broader issue of {\it causal geodesic connectedness}, i.e., the question of when two points $p,q\in M$ which can be connected by a causal curve segment can also be connected by a causal geodesic. After all, a timelike geodesic loop connects a point to itself in this way.

Although causal geodesic connectedness is a classic problem in Lorentzian geometry, there are few general results which do not assume either global hyperbolicity or larger symmetry, but we have recently \cite{costaesilva&flores&honorato21} been able to make some new contributions to this issue. For the remainder of this section we shall state a few results and concepts we have found useful in \cite{costaesilva&flores&honorato21} insofar as they shed light on our concerns here. We indicate herein the numbers of definitions and theorems in that reference for the convenience of the reader interested in looking up the proofs (which we do not reproduce here), but we have slightly altered the notation and scope to suit our present needs. 

\begin{defi}[Def. 5.1 in \cite{costaesilva&flores&honorato21}]\label{def1}
Let $p\in M$. We say that $\exp _p$ has the {\it causal continuation property} (CCP) if for any (piecewise smooth) causal curve $\beta:[0,1] \rightarrow M$ with $\beta(0)=p$, and for any continuous curve $\overline{\beta}:[0,a)\subset [0,1] \rightarrow \overline{\mathcal{T}_p}$ such that $\overline{\beta}(0)=0_p$ and 
$$\exp_p\circ \overline{\beta} = \beta|_{[0,a)}$$
there exists a sequence $(t_k)_{k \in \mathbb{N}}\subset [0,a)$ with $t_k\rightarrow a$ for which $\overline{\beta}(t_k)$ converges to some $\overline{x}\in \mathcal{D}_p$ (and thus $\overline{x} \in \overline{\mathcal{T}_p}  $). 
\end{defi}

The previous notion allows one to obtain sufficient conditions to establish the existence of a {\it timelike} geodesic in $M$ from the point $p$ to another point $q$. Here and hereafter, we denote by $Conj_c(p)$ the set of conjugate points to $p$ along {\it causal} geodesics starting at $p$. 

\begin{thm}[Thm. 5.5 in \cite{costaesilva&flores&honorato21}]\label{thmclave3}
Let $(M, g)$ be a Lorentz manifold and $p \in M$; assume that $\exp _p$ has the CCP. Let $q\in M$ and assume there exists a timelike curve $\alpha:[0,1] \rightarrow M$ with $\alpha(0) =p$, $\alpha(1) =q$ which does not intersect $Conj_c(p)$. Then there exists a timelike geodesic from $p$ to $q$. In particular, if $p=q$ so that $\alpha$ is a timelike loop, then there exists a timelike geodesic loop at $p$. 
\end{thm}
\qcd
\begin{rem}\label{remflaherty}
The use of the CCP is shown to be mandatory in the cited reference, but we suspect - although we have not been able to prove yet - that the ban on conjugate points along causal geodesics may be an artifact of our particular approach. In any case, the following well known geometric condition ensures they are altogether absent from $(M,g)$. Suppose that for any $p \in M$ and for any {\it timelike plane} $\Pi_p \subset T_pM$, the sectional curvatures $K(\Pi_p)\geq 0$. Then for any $p \in M$, we have $Conj_c(p) =\emptyset$ by \cite[Prop. 2.1]{Flaherty1} (cf. also \cite[Prop. 11.13]{bookbeem96}). 
\end{rem}

In view of the somewhat technical character of the condition that the exponential map has the CCP in Thm. \ref{thmclave3}, it is of interest to give a natural geometric condition which entails it. It turns out that {\it causal pseudoconvexity and causal disprisonment} are natural such conditions. Recall that a collection $\mathcal{C}$ of (non-constant) geodesics on $(M,g)$ is 
\begin{itemize}
\item[a)] {\it pseudoconvex} if for any compact set $K\subset M$ there exists a compact set $K^{\ast}\subset M$ such that any segment of a geodesic in $\mathcal{C}$ with endpoints in $K$ is entirely contained in $K^{\ast}$;
\item[b)] {\it disprisoning} if for any given maximal extension $\gamma:(a,b)\rightarrow M$ of a geodesic in $\mathcal{C}$ ($-\infty\leq a<b\leq +\infty$), and any $t_0\in (a,b)$, neither $\overline{\gamma[t_0,b)}$ nor $\overline{\gamma (a,t_0]}$ is compact. If $\mathcal{C}$ is not disprisoning, then it is said to be {\it imprisoning}. 
\end{itemize}
If we take $\mathcal{C}$ to be the collection of all causal geodesics and if it is pseudoconvex/disprisoning, then $(M,g)$ itself is said to be {\it causally pseudoconvex}/{\it causally disprisoning}, respectively. In this paper, by the term {\it causally pseudoconvex and disprisoning} we always mean causally pseudoconvex {\em and } causally disprisoning. 

It is well-known (cf., e.g., \cite[Prop. 7.36]{bookbeem96}) that if a spacetime $(M,g)$ is globally hyperbolic, then it is causally pseudoconvex and disprisoning. The converse, however, is false. This is illustrated by the strip $\{(t,x) \in \mathbb{R}^2 \, : \, 0<x<1\}$ (with the restricted metric) in the Minkowski plane $(\mathbb{R}^2,-dt^2 +dx^2)$. In other words, causal pseudoconvexity and disprisonment is a strictly weaker requirement on $(M,g)$ or its coverings than global hyperbolicity. 

The incorporation of causal pseudoconvexity and disprisonment in the problem of causal connectedness in general and the existence of timelike geodesic loops in particular can summarized as follows (see discussion around Corollary 5.8 in \cite{costaesilva&flores&honorato21}). 

\begin{prop}\label{coolcor} Let $(M,g)$ be a Lorenztian manifold such that $Conj_c(p) =\emptyset$ for any $p \in M$ (which in particular occurs if $(M,g)$ has non-negative sectional curvatures on timelike planes, cf. Remark \ref{remflaherty}). If $(M,g)$ admits a Lorentzian covering $\phi:(\hat{M},\hat{g})\rightarrow (M,g)$ such that $(\hat{M},\hat{g})$ is causally pseudoconvex and disprisoning, then for any $p,q \in M$ which can be connected by a timelike curve, there exists a timelike geodesic connecting $p$ and $q$. Thus, if in addition $(M,g)$ is compact, then there exists a timelike geodesic loop on $(M,g)$. 
\end{prop}
\qcd

\section{The space of timelike geodesic loops}\label{sec1.1}

As justified in the previous sections, we wish to deal directly with the space of timelike geodesic loops, and develop variational arguments directly thereon. To this end, we describe here a convenient parametrization for it. We also obtain here our first important result, which we shall informally refer to as the {\it Hill-Climb Lemma}. This will be crucial later on to establish the existence of closed geodesics, but it has independent interest.

In order to obtain a convenient parametrization of the space of timelike geodesic loops, consider the smooth map $E:\mathcal{D} \subset TM \rightarrow M \times M$ given by $E(v):=(\pi(v),\exp(v))$ for $v \in \mathcal{D}$, where $\pi$ is the natural projection of the tangent bundle $TM$ onto $M$. Let 
$$\mathcal{D}_{\mathcal{T}}:= \mathcal{D}\cap \left(\bigcup _{p\in M}\mathcal{T}_p\right)$$
be the set of al timelike vectors in the domain of the exponential map. Observe that $\mathcal{D}_{\mathcal{T}}$ is still open in $TM$, and consider the restriction $E_{\mathcal{T}}:= E|_{\mathcal{D}_{\mathcal{T}}}$.

If we denote by $\Delta_M\subset M\times M$ the diagonal in $M\times M$, i.e., the set of all points of the form $(p,p)$, then the set of timelike geodesic loops is in bijective correspondence with
\begin{equation}\label{newcooleq}
\mathcal{L} := E_{\mathcal{T}}^{-1}(\Delta_M),
\end{equation}
provided one regards all geodesic loops as affinely parametrized\footnote{This is not necessarily the most convenient affine parametrization for geodesics in all circumstances, and in some calculations we shall use another in which velocities are timelike unit vectors. The description can be easily be modified to reflect this. } on the interval $[0,1]$, and a given timelike geodesic loop $\gamma: t\in [0,1]\mapsto \exp(t\cdot v) \in M$ is identified with its initial velocity $v$. We then endow the set of timelike geodesic loops $\gamma:[0,1] \rightarrow M$ with the unique topology for which the correspondence $\gamma \mapsto \dot{\gamma}(0) \in \mathcal{L}$ is a homeomorphism (for $\mathcal{L}$ with the subspace topology inherited from $TM$). This is the space of timelike geodesic loops.
\begin{rem}\label{callitup}
The following features of $\mathcal{L}$ are not difficult to ascertain.
\begin{itemize}
    \item[1)] Since $\Delta_M$ is the graph of the identity function on $M$, it is of course a closed $n$-dimensional submanifold of $M\times M$. Thus, $\mathcal{L}$ is closed in $\mathcal{D}_{\mathcal{T}}$.
    \item[2)] $\mathcal{L}$ does not need to be a submanifold of $TM$, but it does have an $n$-dimensional submanifold structure around points of $\mathcal{L}$ where $E_{\mathcal{T}}$ is nonsingular. Note that $E$ is nonsingular at $v\in \mathcal{D}_{\mathcal{T}} $ whenever $\exp_{\pi(v)}$ is nonsingular. Geometrically, the latter sort of singular point corresponds to a timelike geodesic loop $\gamma:[0,1]\rightarrow M$ which is {\it self-conjugate}, that is, $\gamma(1)$ is conjugate to $\gamma(0)=\gamma(1)$ along $\gamma$. Restricted to a smooth patch where such self-conjugate timelike geodesic loops are absent, $E_{\mathcal{T}}$ gives a local diffeomorphism between (a local smooth patch of) $\mathcal{L}$ and $\Delta_M$.
    \item[3)] The length functional on timelike geodesic loops corresponds on $\mathcal{L}$ to the map $v\in \mathcal{L}\mapsto |v| \in \mathbb{R}$. This is evidently continuous with respect to the induced topology on $\mathcal{L}$, so the length functional is continuous on the space of timelike geodesic loops.
    \item[4)] Let $\gamma_1,\gamma_2:[0,1]\rightarrow M$ be two timelike geodesic loops. If there exists a continuous curve $\beta:[0,1] \rightarrow \mathcal{L}$ connecting $\dot{\gamma}_i(0)$ ($i=1,2$), then 
    $$\sigma(s,t) := \exp(t\cdot \beta(s)) \in M$$
    is a $TG$-homotopy between $\gamma_1$ and $\gamma_2$. {\it Therefore, two timelike geodesic loops in the same path-connected component are $TG$-homotopic.} (Conversely, if there exists a $C^1$ $TG$-homotopy $\lambda:[0,1]^2 \rightarrow M$ between the $\gamma_i$, then there is a continuous curve in $\mathcal{L}$ connecting the $\dot{\gamma}_i(0)$.) 
    \item[5)] We emphasize that the space of timelike geodesic loops introduced here is {\em not} the same as (a subset of) the space of geodesics defined in \cite{beem&parker91}. To see why, let $\gamma:[a,b] \rightarrow M$ be a timelike geodesic segment. Suppose $\gamma$ self-intersects (at least) twice, say $\gamma(a) = \gamma(b) = \gamma(c)$ with $a<c<b$. Then the timelike geodesic loops $\eta_1,\eta_2,\eta_3:[0,1] \rightarrow M$ given, for $0\leq t\leq 1$, by
    \begin{eqnarray}
    \eta_1(t) &:=& \gamma((c-a).t + a), \nonumber \\
    \eta_2(t) &:=& \gamma( (b-c).t + c), \nonumber \\
    \eta_3(t) &:=& \gamma((b-a).t + a), \nonumber
    \end{eqnarray}
    which in general are all distinct, although they all give rise to a single inextendible geodesic up to affine reparametrization. 
    
\end{itemize}

\end{rem}

\medskip

\begin{prop}\label{closedness}
On the Lorentzian manifold $(M,g)$, let $(\eta_k:[0,1]\rightarrow M)_{k\in \mathbb{N}}$ be a sequence of timelike geodesic loops, all of which are in the same $TG$-homotopy class $\mathfrak{TG}(\gamma)$ of some timelike geodesic loop $\gamma$. Assume 
$$\dot{\eta}_k(0) \rightarrow v \quad \mbox{ in $TM$},$$
where $v \in \mathcal{D}_{\mathcal{T}}$ is a timelike vector, and assume 
$$\eta: t\in [0,1] \mapsto \exp(t\cdot v)\in M$$ is a non-self-conjugate timelike geodesic loop. Then $\eta \in \mathfrak{TG}(\gamma)$. 
\end{prop}

\noindent {\it Proof.} In order to see that $\eta \in \mathfrak{TG}(\gamma)$, it suffices to show that it is in the $TG$-homotopy class of any of the $\eta_k$. To that end, write $v_k := \dot{\eta}_k(0)\in T_{p_k}M$ with basepoint $p_k \in M$, for each $k$, and let $p$ be the basepoint of $v$. Then
$$E(v_k) \rightarrow E(v) \Rightarrow p_k \rightarrow p, $$
and since $\mathcal{L}$ is closed in $\mathcal{D}_{\mathcal{T}}$ (cf. Remark \ref{callitup}(1)), $v \in \mathcal{L}$, and thus $\eta$ is indeed a timelike geodesic loop.

Since $\eta$ is non-self-conjugate, $(d\exp_p)_{v}$ is nonsingular, and therefore $E$ is not singular at $v$. By the Inverse Mapping theorem there are connected open sets $A\ni v$ and $B\in p$ respectively in $\mathcal{D}_{\mathcal{T}}$ and $M$, respectively, such that $E\vert_A:A\rightarrow B\times B$ is a diffeomorphism. (In particular, all vectors of $A$ are timelike.) As $\lim v_k=v$, we can pick some $v_{k_0}\in A$ and, consequently, $E(v_{k_0})=(p_{k_0},p_{k_0})\in B\times B$. Now, let $\alpha:[0,1]\rightarrow M$ be any smooth curve from $p$ to $p_{k_0}$ in $B$. Thus, $\bar{\alpha}:=(E\vert_A)^{-1}\circ(\alpha\times\alpha):[0,1]\rightarrow TM$ is a smooth curve that connects $v$ to $v_{k_0}$ in $A$. Define the smooth map
$$
\begin{array}{cccc}
\lambda \ : & \! [0,1]\times [0,1] & \! \longrightarrow & \! M \\
& \! (s,t) & \! \longmapsto & \! \lambda(s,t)=\exp\left(t\cdot \bar{\alpha}(s)\right)
\end{array}.
$$
We have that
\begin{itemize}
\item $\lambda(0,t)=\exp\left(t\cdot \bar{\alpha}(0)\right)=\exp_p(tv)=\eta(t)$, for all $t\in[0,1]$;
\item $\lambda(1,t)=\exp\left(t\cdot \bar{\alpha}(1)\right)=\exp_{p_{k_0}}(tv_{k_0})=\eta_{k_0}(t)$, for all $t\in[0,1]$;
\item The longitudinal curves $\lambda_s$ are timelike (by the choice of the open set $A$ above), and besides,
$$\lambda_s(0)=\exp\left(0\cdot \bar{\alpha}(s)\right)=\alpha(s)=\exp\left(1\cdot \bar{\alpha}(s)\right)=\lambda_s(1).$$
I.e., the longitudinal curves $\lambda_s$ are timelike geodesic loops.
\end{itemize}
Therefore, $\lambda$ is a $TG$-homotopy deforming $\eta$ onto $\eta_{k_0}$ via timelike geodesic loops and, hence, $\eta\in \mathfrak{TG}(\eta_{k_0})=\mathfrak{TG}(\gamma)$ as desired. (Compare with Remark \ref{callitup}(4).)

\qcd


The major result of this section is the announced Hill-Climb Lemma. It entails that timelike geodesic loops can have their lengths either enlarged or shortened within the same $TG$-homotopy class {\it unless} they are {\it either} self-conjugate, {\it or} closed timelike geodesics. 

\medskip

\begin{thm}\label{stretch}
Let $\gamma:[0,l]\rightarrow M$ be a non-self-conjugate timelike geodesic loop on the Lorentzian manifold $(M,g)$. Then, either $\gamma$ is a closed timelike geodesic, or else there exists a smooth $TG$-homotopy $\sigma:[-\delta, \delta] \times [0,l] \rightarrow M$ such that $\sigma_0 \equiv \gamma$, and
$${\rm length}(\sigma _{s}) < {\rm length}(\gamma)<{\rm length}(\sigma _{s'}) , \quad\hbox{whenever $-\delta\leq s<0<s'\leq\delta$}.$$ 
\end{thm}

\noindent {\it Proof.}  
We may assume without loss of generality that $\gamma:[0,l]\rightarrow M$ is a unit timelike geodesic segment with $\gamma(0)=\gamma(l)=:p$, and that $\gamma(0)(=\gamma(l))$ is not conjugate to itself along $\gamma$. Recall that we denote the exponential map on $(M,g)$ by $\exp :\mathcal{D}\subset TM\rightarrow M$, and for any $q\in M$, we write $\mathcal{D}_q := \mathcal{D}\cap (T_qM)$ and $\exp _q:= \exp |_{\mathcal{D}_q}$. Thus, $v_0:= l\, \dot{\gamma}(0) \in \mathcal{D}_p$ and  $\exp_p(v_0)=\exp_{\gamma(0)}(l\dot{\gamma}(0))=\gamma(l)=p$. 

The condition that $\gamma$ is not self-conjugate implies that $(d\exp_p)_{v_0}:T_{v_0}(T_p M)\rightarrow T_p M$ is non-singular. Let $E:\mathcal{D} \subset TM \rightarrow M \times M$ given by $E(v):=(\pi(v),\exp(v))$ for $v \in \mathcal{D}$, where $\pi$ is the natural projection of $TM$ onto $M$. We know that $E$ is nonsingular at $v_0 \in \mathcal{D} $ since $\exp_p$ is nonsingular, too, and also that
\begin{equation*}
E(v_0)=(\pi(v_0),\exp(v_0))=(p,p).
\end{equation*} 
By the Inverse Function Theorem, there exist connected open sets $U\ni  v_0$ and $V \ni p$ of $TM$ and $M$, respectively, for which $E|_{U} : U \subset \mathcal{D} \rightarrow V \times V \subset M \times M$ is a diffeomorphism.  Shrinking $U$ and $V$ if necessary we can (and will) without loss of generality assume that every vector $v \in U$ is timelike\footnote{Note that $v_0 \in TM$ is itself timelike, which is an open condition.} and $V$ is a normal convex neigborhood of $p$. We write $U_q := U\cap \mathcal{D}_q$ for each $q\in V$, and one easily checks that $\exp_q|_{U_q}:U_q \rightarrow V$ is a diffeomorphism.

Fix $0< \vartheta< l$ such that $\gamma[-\vartheta,\vartheta] \subset V$\footnote{Of course, we have here implicitly extended $\gamma$ slightly to the left of $0$, but we shall retain its name.} and let $\alpha:=\gamma|_{(-\vartheta,\vartheta)}$. Since $V$ is convex, we have $\alpha(t) \neq \alpha(s)$ for $s\neq t$ in $(-\vartheta,\vartheta)$. Now, for each $q\in V$ define the smooth curve
\[
\overline{\alpha}_q:=(\exp_q\mid_{U_q})^{-1}\circ\alpha :(-\vartheta,\vartheta)\rightarrow U_q\subset T_qM.
\]
 (Note that we always have $r_q(t):= |\overline{\alpha}_q(t)|>0$ for any $q\in V$ and any $t\in (-\vartheta,\vartheta)$ by our choice of $U$.) In addition, let  
 \[
 w_q(t):= \frac{\overline{\alpha}_q(t)}{r_q(t)}=\frac{\overline{\alpha}_q(t)}{|\overline{\alpha}_q(t)|} \quad \forall t \in (-\vartheta,\vartheta),
 \]
 and
 \[
 \mathcal{O}_q=\{(r,t) \in (0,+\infty)\times (-\vartheta,\vartheta) \, : \, r\cdot w_q(t)  \in U_q\}.
 \]
 We then define  
 \begin{equation}\label{mm}
f_q: (r,t) \in \mathcal{O}_q \mapsto \exp_q(r\, w_q(t)) \in M.
\end{equation}
In particular,
\begin{equation}\label{eq1}
\dot{\alpha}(t)=\frac{d}{dt}\exp_q(\overline{\alpha}_q(t))=\frac{d}{dt}f_q(r_q(t),t)=\frac{\partial f_q}{\partial r}\dot{r}_q(t)+\frac{\partial f_q}{\partial t}.
\end{equation}
By Gauss' Lemma, $g(\partial f_q/\partial r,\partial f_q/\partial t)=0$. Moreover, $g(\partial f_q/\partial r,\partial f_q/\partial r)=-1$. In a Lorentzian manifold these relations imply that $g(\partial f_q/\partial t,\partial f_q/\partial t)\geq 0$ (i.e., $\partial f_q/\partial t$ is spacelike).
Hence,
$$-1=g(\dot{\alpha}(t),\dot{\alpha}(t))=g(\partial f_q/\partial r,\partial f_q/\partial r)\cdot\dot{r}_q(t)^2+g(\partial f_q/\partial t,\partial f_q/\partial t)=-\dot{r}_q(t)^2+\left|\frac{\partial f_q}{\partial t}\right|^2,$$
or 
\begin{equation}\label{key1}
|\dot{r}_q(t)|= \sqrt{1+ \left|\frac{\partial f_q}{\partial t}\right|^2}= 1 + \frac{\left|\frac{\partial f_q}{\partial t}\right|^2}{1+ \sqrt{1+ \left|\frac{\partial f_q}{\partial t}\right|^2}}.
\end{equation}
In particular, $|\dot{r}_q(t)|\geq 1$, and hence $\dot{r}_q(t)$ is either always positive or always negative for every $(t,q) \in (-\vartheta,\vartheta)\times V$. 

Now, observe that 
\begin{equation}\label{comp}
\overline{\alpha}_q(t) = (E|_{U})^{-1}(q,\alpha(t)), \quad \forall (t,q) \in (-\vartheta,\vartheta)\times V,
\end{equation}
and hence $r_q(t)=|\overline{\alpha}_q(t)|$ varies smoothly on $(-\vartheta,\vartheta)\times V$; by continuity there exist an open set $V_0\ni p$ contained in $V$ and a number $0<\vartheta_0 < \vartheta$ such that $\dot{r}_q(t)$ has the same sign as $\dot{r}_p(0)$, whenever $(t,q)\in (-\vartheta_0, \vartheta_0)\times V_0$.  

For the rest of this proof we shall refer to a positive number $\delta \in (0,\vartheta_0)$ as {\it acceptable} if it is small enough that $v_{\pm\delta}:=(l\mp\delta)\dot{\gamma}(\pm \delta) \in U$ (note that $\lim_{\delta \rightarrow 0} v_{\pm\delta}=v_0\in U$), and $\gamma[-\delta,\delta] \subset V_0$. Fix any such an acceptable $\delta>0$, let $q_{\pm\delta}:=\gamma(\pm\delta)\neq p$. We compute 
\begin{equation}\label{g}
\exp_{q_{\pm\delta}}(v_{\pm\delta}) = \exp_{\gamma(\pm\delta)}((l\mp\delta)\dot{\gamma}(\pm\delta)) = \gamma(\pm\delta + (l\mp\delta)) = \gamma(l) = p,
\end{equation}
From (\ref{g}) we deduce that
\begin{equation}\label{g1}
\overline{\alpha}_{q_{\pm\delta}}(0)=(\exp_{q_{\pm\delta}}\mid_{U_{q_{\pm\delta}}})^{-1}(\alpha(0))=(\exp_{q_{\pm\delta}}\mid_{U_{q_{\pm\delta}}})^{-1}(p)=v_{\pm\delta}\in U_{q_{\pm\delta}}\subset \mathcal{D}_{q_{\pm\delta}}.    
\end{equation}
So, if we define
\begin{equation}\label{g2}
u_{\pm\delta}:=\overline{\alpha}_{q_{\pm\delta}}(\pm\delta)=(\exp_{q_{\pm\delta}}\mid_{U_{q_{\pm\delta}}})^{-1}(\alpha(\pm\delta))\in U_{q_{\pm\delta}}\subset \mathcal{D}_{q_{\pm\delta}},
\end{equation}
the vectors $u_{\pm\delta}, v_{\pm\delta} \:(\neq 0)$ are distinct (recall that $\alpha(0)=p\neq q_{\pm\delta}=\alpha(\pm\delta)$ and $\exp_{q_{\pm\delta}}\mid_{U_{q_{\pm\delta}}}$ is a diffeomorphism).

On the one hand, by (\ref{comp}) we have
\[
u_{\pm\delta} \stackrel{\delta}{\longrightarrow} \overline{\alpha}_p(0),
\]
and 
\[
\exp _p(v_0) = p= \alpha(0)= \exp _p(\overline{\alpha}_p(0)) \Rightarrow \overline{\alpha}_p(0) = l \, \dot{\gamma}(0),
\]
whence we conclude that 
\begin{equation}\label{d1}
u_{\pm\delta}/l \stackrel{\delta}{\longrightarrow} \dot{\gamma}(0).
\end{equation}
Now, if we define the geodesics
\[
\beta_{\pm \delta}: t\in [0,l] \mapsto \exp _{q_{\pm \delta}}(t\cdot (u_{\pm \delta}/l)) \in M
\]
we see easily that $\beta_{\pm \delta}(0)= \beta_{\pm \delta }(l)= q_{\pm \delta}$, i.e., these are geodesic loops. Moreover, $\dot{\beta}_{\pm \delta}(0) = u_{\pm \delta}/l$. By (\ref{d1}), the initial velocity of these geodesic loops can made as close to $\dot{\gamma}(0)$ as desired by taking $\delta \rightarrow 0$, and 
\begin{equation}\label{lengthpm}
{\rm length}(\beta _{\pm \delta}) = |u_{\pm \delta}|.
\end{equation}

There are now two cases to consider. \vspace{0.2cm}\\
\uline{\textsc{$1^{\uline{\textrm{st}}}$ Case:}}  $\dot{r}_p(0) <0$. \vspace{0.2cm}

Our previous choices then imply that $\dot{r}_{q}(t) <0$ for any $q\in V_0$ and $t\in (-\vartheta_0,\vartheta_0)$. In particular, if we pick an acceptable $\delta >0$ we can integrate Eq. (\ref{key1}) on $[-\delta,0]$ with $q=q_{-\delta}$ to obtain
\begin{equation}\label{key2}
\int_{-\delta} ^0|\dot{r}_{q_{-\delta}}(t)|dt = -(r_{q_{-\delta}}(0) - r_{q_{-\delta}}(-\delta)) = \delta + \eta _{-\delta} \Rightarrow |u_{-\delta}|-|v_{-\delta}| = \delta + \eta_{-\delta},
\end{equation}
where we have defined
\begin{equation}\label{key3}
\eta_{-\delta} := \int_{-\delta}^0\frac{\left|\frac{\partial f_{q_{-\delta}}}{\partial t}\right|^2}{1+ \sqrt{1+ \left|\frac{\partial f_{q_{-\delta}}}{\partial t}\right|^2}} \, dt\geq 0.
\end{equation}
Therefore, using the definition of $v_{-\delta}$ and (the discussion around) (\ref{lengthpm}) in Eq.(\ref{key2}) yields
\begin{equation}\label{key4}
{\rm length}(\beta _{- \delta}) = l + 2\delta +\eta _{-\delta} > l \equiv {\rm length}(\gamma),
\end{equation}
Thus, if we define 
$$\sigma: (s,t) \in [0,\delta]\times [0,l] \mapsto \beta_{-s}(t) \in M$$
we have a $TG$-homotopy of $\gamma$ for which each longitudinal curve (with $s>0$) has strictly larger Lorentzian length than $\gamma$.

On the other hand, we can also integrate Eq. (\ref{key1}) on $[0,\delta]$ with $q=q_{+\delta}$ to obtain
\begin{equation}\label{key2'}
	\int^{\delta} _0|\dot{r}_{q_{+\delta}}(t)|dt = -(r_{q_{+\delta}}(\delta) - r_{q_{+\delta}}(0)) = \delta + \eta _{+\delta} \Rightarrow -|u_{+\delta}|+|v_{+\delta}| = \delta + \eta_{+\delta},
\end{equation}
where
\begin{equation}\label{key3'}
	\eta_{+\delta} := \int^{\delta}_0\frac{\left|\frac{\partial f_{q_{+\delta}}}{\partial t}\right|^2}{1+ \sqrt{1+ \left|\frac{\partial f_{q_{+\delta}}}{\partial t}\right|^2}} \, dt\geq 0.
\end{equation}
Thus, using the definition of $v_{+\delta}$ and (the discussion around) (\ref{lengthpm}) in Eq.(\ref{key2'}) yields
\begin{equation}\label{key4'}
	{\rm length}(\beta _{+ \delta}) = l - 2\delta -\eta _{+\delta} < l \equiv {\rm length}(\gamma).
\end{equation}
Therefore, if we now define 
$$\sigma: (s,t) \in [0,\delta]\times [0,l] \mapsto \beta_{s}(t) \in M,$$
we have a $TG$-homotopy of $\gamma$ whose longitudinal curves with $s>0$ have strictly {\it shorter} Lorentzian length than that of $\gamma$.
\vspace{0.2cm}\\
\uline{\textsc{$2^{\uline{\textrm{nd}}}$ Case:}} $\dot{r}_p(0) >0$. \vspace{0.2cm}

In this case, $\dot{r}_{q}(t) >0$ for any $q\in V_0$ and $t\in (-\vartheta_0,\vartheta_0)$. Again we pick any acceptable $\delta >0$, but we now integrate Eq. (\ref{key1}) on $[0,\delta]$ with $q=q_{+\delta}$:
\begin{equation}\label{key5}
\int ^{\delta} _0|\dot{r}_{q_{+\delta}}(t)|dt = (r_{q_{+\delta}}(\delta) - r_{q_{+\delta}}(0)) = \delta + \eta _{+\delta} \Rightarrow |u_{+\delta}|-|v_{+\delta}| = \delta + \eta_{+\delta},
\end{equation}
where we have now defined
\begin{equation}\label{key6}
\eta_{+\delta} := \int^{+\delta}_0\frac{\left|\frac{\partial f_{q_{+\delta}}}{\partial t}\right|^2}{1+ \sqrt{1+ \left|\frac{\partial f_{q_{+\delta}}}{\partial t}\right|^2}} \, dt \geq 0.
\end{equation}
However, this time around,  using the definition of $v_{+\delta}$ and (the discussion around) (\ref{lengthpm}) in Eq.(\ref{key5}), 
we obtain
\begin{equation}\label{key7}
{\rm length}(\beta _{+ \delta}) = l  +\eta _{+\delta} \geq l \equiv {\rm length}(\gamma).
\end{equation}
This inequality is not {a priori} strict, but there are two possibilities: {\it either} $(i)$ there exists some acceptable $\delta_0>0$ such that $\eta _{+\delta}>0$ for {\it every} $0<\delta \leq \delta_0$, and if so we are back to strict inequalities in (\ref{key7}) on this range; then, arguing as in the first case, a $TG$-homotopy of $\gamma$ can be found for which {\it every } longitudinal curves with $s>0$ strictly increase length with respect to $\gamma$; or {\it else} $(ii)$ for {\it some} acceptable $\delta_0>0$, we ought to have $\eta_{+\delta_0} \equiv 0$. In that case, since the integrand in (\ref{key6}) is nonnegative one gets

\begin{equation}\label{key8}
\frac{\partial f_{q_{+\delta_{0}}}}{\partial t}\equiv 0, \quad \hbox{on $[0,\delta_0]$},
\end{equation}
which plugged back into (\ref{key1}) gives (recall that $\dot{r}_{q_{+\delta_0}}>0$) 
\begin{equation}\label{key9}
r_{q_{+\delta_0}}(t) = t+ (l-\delta_0) \quad \forall t\in [0,\delta_0].
\end{equation}
In addition, when (\ref{key8}) is used in (\ref{mm}), it allows us to conclude that 
\begin{equation}\label{key10}
\dot{w}_{+\delta_0} \equiv 0 \Rightarrow w_{+\delta_0}(t) = w_{+\delta_0}(0) = \frac{v_{+\delta_0}}{|v_{+\delta_0}|} \equiv \dot{\gamma}(\delta_0) \quad \forall t \in [0,\delta_0].
\end{equation}
Therefore, (\ref{key9}) and (\ref{key10}) yield
\begin{equation}\label{key11}
\overline{\alpha}_{\gamma(\delta_0)}(t) = (t+ l - \delta_0)\cdot \dot{\gamma}(\delta_0) \Rightarrow \alpha(t) = \exp _{\gamma(\delta_0)}((t+l-\delta_0)\cdot \dot{\gamma}(\delta_0)) \Rightarrow \gamma(t) = \gamma(t+l)  
\end{equation}
on $[0,\delta_0]$, and in particular 
\[
\dot{\gamma}(0)=\dot{\gamma}(l),
\]
which in turn means that $\gamma$ is a closed timelike geodesic. 

Finally, for the last subcase, we integrate 
Eq. (\ref{key1}) on $[-\delta,0]$ with $q=q_{-\delta}$:
\begin{equation}\label{key5'}
	\int _{-\delta} ^0|\dot{r}_{q_{-\delta}}(t)|dt = (r_{q_{-\delta}}(0) - r_{q_{-\delta}}(-\delta)) = \delta + \eta _{-\delta} \Rightarrow |v_{-\delta}|-|u_{-\delta}| = \delta + \eta_{-\delta},
\end{equation}
where 
\begin{equation}\label{key6'}
	\eta_{-\delta} := \int_{-\delta}^0\frac{\left|\frac{\partial f_{q_{-\delta}}}{\partial t}\right|^2}{1+ \sqrt{1+ \left|\frac{\partial f_{q_{-\delta}}}{\partial t}\right|^2}} \, dt\geq 0.
\end{equation}
By using the definition of $v_{-\delta}$ and (the discussion around) (\ref{lengthpm}) in Eq.(\ref{key5'}) the number $\delta$ now cancels out on both sides, and we obtain
\begin{equation}\label{key7'}
	{\rm length}(\beta _{-\delta}) = l  -\eta _{-\delta} \leq l \equiv {\rm length}(\gamma).
\end{equation}
Again, this yields a $TG$-homotopy of $\gamma$ {\it decreasing} its length for small $\delta$, but {\it only} if we know that $\eta _{-\delta}>0$ on an interval of acceptable $\delta$'s. If  not, $\eta_{-\delta_0}=0$ for at least some acceptable $\delta_0>0$, in which case an argument entirely analogous to that of the last steps in the previous subcase shows that $\gamma$ is again a closed timelike geodesic.
 
\qcd

We say that a timelike geodesic loop $\gamma:[0,1]\rightarrow M$ is \textit{locally maximizing} (resp. \textit{locally minimizing}) if $\comp(\zeta) \leq \comp(\gamma)$ (resp. $\comp(\zeta) \geq \comp(\gamma)$) for any timelike geodesic loop $\zeta:[0,1]\rightarrow M$ with initial conditions $(\zeta(0),\dot{\zeta}(0))\in TM$ close\footnote{Here, $TM$ is itself implicitly endowed with a suitable Riemannian metric, say, the Sasaki metric associated with some auxiliary Riemannian metric $h$ on $M$.} enough to $(\gamma(0),\dot{\gamma}(0))$. We say that $\gamma$ is \textit{locally extremal} if it is either locally maximizing or locally minimizing. We then have an immediate consequence of this definition in view of the Hill-Climb Lemma:

\begin{cor}\label{cor1hillclimb}
Any locally extremal, non-self-conjugate timelike geodesic loop in the Lorentzian manifold $(M,g)$ is a closed timelike geodesic. 
\end{cor}
\qcd

\section{A novel existence result for closed timelike geodesics}\label{sec2}

We turn now to the main geometric consequences of the analysis in the previous section. The main results here are Proposition \ref{newmainthm'} and Theorem \ref{newmainthm} below, which establish new natural conditions for the existence of a closed timelke geodesic in a Lorentz manifold. Roughly speaking, the {\it precompactness} (in a suitable sense) of the $TG$-homotopy class of a timelike geodesic loop is the key assumption leading to the existence a locally extremal timelike geodesic loop, Corollary \ref{cor1hillclimb} doing the rest. 

\begin{prop}\label{newmainthm'}
	Suppose the Lorentzian manifold $(M,g)$ contains some timelike geodesic loop $\gamma$ such that (the set of initial velocities of the members of) $\mathfrak{TG}(\gamma)$ is precompact in $\mathcal{D}\subset TM$. 
	\newline \indent Then, $L_{\mathfrak{TG}}(\gamma)<\infty$ and there exists a timelike geodesic loop $\eta$ with length $L_{\mathfrak{TG}}(\gamma)$ such that either $\eta \in \mathfrak{T}(\gamma)$ and it is self-conjugate, or $\eta \in \mathfrak{TG}(\gamma)$ and it is a closed timelike geodesic therein. 
	\newline \indent If in addition $l_{\mathfrak{TG}}(\gamma)>0$, then there exists a timelike geodesic loop $\eta'$ with length $l_{\mathfrak{TG}}(\gamma)$ such that either $\eta' \in \mathfrak{T}(\gamma)$ and it is self-conjugate, or $\eta' \in \mathfrak{TG}(\gamma)$ and it is a closed timelike geodesic therein.  
\end{prop}

\noindent {\it Proof.} 
We only prove the first statement, since the arguments when the infimum  $l_{\mathfrak{TG}}(\gamma)>0$ are entirely analogous.  Let $(\eta_{k}: [0,1] \rightarrow M)_{k\in \mathbb{N}}$ be a sequence of timelike geodesic loops in the $TG$-homotopy class $\mathfrak{TG}(\gamma)$ whose sequence of lengths $(\ell_k=|\dot{\eta}_k(0)|)_{k \in \mathbb{N}}$ approaches $\ell := L_{\mathfrak{TG}}(\gamma)\leq \infty $, i.e., 
	\begin{equation}\label{limit0}
		\ell_k \rightarrow \ell. 
	\end{equation}
	Let $p_k := \eta_k(0)=\eta_k(1)$. The assumed precompactness implies that, up to passing to a subsequence, we have 
	$$\dot{\eta}_k(0) \rightarrow v \in \mathcal{D} \Rightarrow \ell_k = |\dot{\eta}_k(0)| \rightarrow |v| \equiv \ell <\infty.$$
	But the very existence of $\gamma$ already implies $\ell >0$, so $v$ is timelike. If $p \in M$ is the basepoint of $v$, then $p_k\rightarrow p$. Therefore, if we define $\eta:t\in [0,1] \mapsto \exp_p(t\cdot v) \in M$, this is a timelike geodesic loop and it is not hard to see that $\eta\in\mathfrak{T}(\gamma)$. Then, either $\eta$ is self-conjugate, or, by Prop. \ref{closedness} and Theorem \ref{stretch}, $\eta \in \mathfrak{TG}(\gamma)$ is a closed timelike geodesic.

\qcd

The precompactness of some $TG$-homotopy class $\mathfrak{TG}(\gamma)$ used in Proposition \ref{newmainthm'} may not be easy to check in practice. We therefore give here concrete, natural geometric conditions in which it holds. Observe that unlike in Prop. \ref{newmainthm'}, we now need to assume that the underlying manifold $M$ is compact.

\begin{thm}\label{newmainthm}
Let $(M,g)$ be a compact Lorenztian manifold. Assume that 
\begin{itemize}
\item[i)] $(M,g)$ contains some timelike geodesic loop $\gamma$ such that $\mathfrak{T}(\gamma)$ contains no self-conjugate timelike geodesic loops; 
\item[ii)] there exists a regular Lorentzian covering $\phi:(\hat{M},\hat{g})\rightarrow (M,g)$ such that $(\hat{M},\hat{g})$ is causally pseudoconvex and disprisoning (which occurs, e.g., if $(\hat{M},\hat{g})$ is globally hyperbolic). 
\end{itemize}
Then, $\mathfrak{TG}(\gamma)$ contains a closed timelike geodesic of length $L_{\mathfrak{TG}}(\gamma)$. If, in addition, $l_{\mathfrak{TG}}(\gamma)>0$, then $\mathfrak{TG}(\gamma)$ also contains a closed timelike geodesic of length $l_{\mathfrak{TG}}(\gamma)$.  
\end{thm}

\noindent {\it Proof.} 
The idea is to use condition (ii) (in combination with $l_{\mathfrak{TG}}(\gamma)>0$ for the last statement) to obtain the needed precompactness in Prop. \ref{newmainthm'}, and then to use condition (i) to discard the self-conjugate geodesic timelike loop in the thesis of that proposition. 

In other words, it suffices to show that the set of initial velocities of elements of $\mathfrak{TG}(\gamma)$ is precompact in $\mathcal{D}$.

Let $(\eta_{k}: [0,1] \rightarrow M)_{k\in \mathbb{N}}$ be a sequence of timelike geodesic loops in the $TG$-homotopy class $\mathfrak{TG}(\gamma)$. Let $p_k := \eta_k(0)=\eta_k(1)$. Since $M$ is compact we may assume, up to passing to a subsequence, that $p_k \rightarrow p$. 

Pick an evenly covered, connected neighborhood $U\ni p$ in $M$, and write 
$$\bigsqcup _{i\in I}\hat{U}_i = \phi^{-1} (U).$$
Choose an open set $V\ni p$ with (compact) closure $\overline{V}\subset U$, some $i_0 \in I$, and $\hat{p} \in \hat{U}_{i_0}$ such that $\phi(\hat{p}) =p$. Eventually $p_k \in V$, so we may assume this is always the case, and consider the lifts $\hat{\eta}_k: [0,1] \rightarrow \hat{M}$ of the geodesics $\eta_k$ starting at $\hat{p}_k := (\phi|_{\hat{U}_{i_0}})^{-1}(p_k) \rightarrow \hat{p}$. 

We have $\hat{\eta}_k(0) \in \hat{K}_{i_0} := (\phi|_{\hat{U}_{i_0}})^{-1}(\overline{V}) $ and 
$\hat{\eta}_k(1) \in \hat{K}_{i_k} := (\phi|_{\hat{U}_{i_k}})^{-1}(\overline{V}) $ for some $i_k\in I$ (possibly $i_k=i_0$), where $\hat{K}_{i_0},\hat{K}_{i_k}$ are thus compact sets. 

Let a pair $k,k'\in \mathbb{N}$ be given. We claim that $U_{i_k}=U_{i_{k'}}$, and therefore in particular we have $i_k=i_{k'}$ and $\hat{K}_{i_k}=\hat{K}_{i_{k'}}$. To see this, we use the fact that $\eta_k,\eta_{k'}$ are in the same $TG$-homotopy class and that the covering is regular. Indeed, let $\sigma_{k,k'}: [0,1]\times [0,1] \rightarrow M$ be a $TG$-homotopy with $\sigma_{k,k'}(0,t) = \eta_k(t)$ and $\sigma_{k,k'}(1,t) =\eta _{k'}(t)$ for every $t \in [0,1]$. Let $\hat{\sigma}_{k,k'}:[0,1] \times [0,1] \rightarrow \hat{M}$ be the (unique) lift of $\sigma_{k,k'}$ through $\phi$ with $\hat{\sigma}_{k,k'}(0,0) = \hat{p}_k$. (This is of course a  homotopy on $(\hat{M},\hat{g})$ between $\hat{\eta}_k$ and $\hat{\eta}_{k'}$ whose longitudinal curves are timelike geodesics, though not necessarily loops.)

By the regularity of the covering there exist deck transformations $F_k, F_{k'}:\hat{M} \rightarrow \hat{M}$ with $F_k(\hat{\eta}_k(0)) = \hat{\eta}_k(1)$ and $F_{k'}(\hat{\eta}_{k'}(0)) = \hat{\eta}_{k'}(1)$. Now, consider the curve $$\hat{\zeta}_{k,k'}: s\in [0,1] \mapsto \hat{\sigma}_{k,k'}(s,0) \in \hat{M}$$
which connects $\hat{p}_k$ and $\hat{p}_{k'}$. One easily checks that the curves $F_k\circ \hat{\zeta}_{k,k'}$ and $$\hat{\xi}_{k,k'}:s\in [0,1] \mapsto \hat{\sigma}_{k,k'}(s,1)\in \hat{M}$$
are both lifts through $\phi$ starting at $\hat{\eta}_k(1)$ of the same curve 
$$s\in [0,1] \mapsto \sigma_{k,k'}(s,0) =\sigma_{k,k'}(s,1) \in M$$
spanning the base points of the longitudinal geodesics of $\sigma_{k,k'}$. We deduce that $F_k\circ \hat{\zeta}_{k,k'}\equiv \hat{\xi}_{k,k'}$, and in particular $F_k\circ \hat{\zeta}_{k,k'}(1)=\hat{\xi}_{k,k'}(1)$. Therefore,
$$F_k(\hat{p}_{k'}) = F_k(\hat{\eta}_{k'}(0)) = F_k(\hat{\zeta}_{k,k'}(1)) = \hat{\xi}_{k,k'}(1) = \hat{\eta}_{k'}(1) = F_{k'}(\hat{\eta}_{k'}(0)) = F_{k'}(\hat{p}_{k'}).$$
Whence we conclude that $F_{k}= F_{k'}$, which in turn establishes that $U_{i_k}=U_{i_{k'}}$ as claimed.

It now follows from the previous claim that all lifts $\hat{\eta}_k$ (which are timelike geodesics in $(\hat{M},\hat{g})$) have endpoints on the compact set 
$$\hat{K} := \hat{K}_{i_0}\cup \hat{K}_{i_1} \subset \hat{M}.$$ 
If we now write 
$$\hat{\eta}_k(1) = \hat{\exp}_{\hat{p}_k}(\dot{\hat{\eta}}_k(0)),$$
then causal pseudoconvexity and disprisonment mean that the exponential map $\hat{\exp}$ restricted to causal vectors is {\it proper}, i.e., inverse images of compact sets in $\hat{M}$ - here, $\hat{K}$, are compact in $\hat{\mathcal{D}}$ \cite[Thm 3.5]{costaesilva&flores&honorato21}. Therefore, again up to passing to a subsequence,
\begin{equation}\label{ohmy} \dot{\hat{\eta}}_k(0)\rightarrow \hat{v} \in \hat{\mathcal{D}}_{\hat{p}} .\end{equation}
Applying $d\phi$ to (\ref{ohmy}) also yields 
$$\dot{\eta}_k(0)\rightarrow d\phi(\hat{v}) =: v \in \mathcal{D}_p.$$
This establishes the desired precompactness, and the conclusion now follows from Prop. \ref{newmainthm'}\qed

\section{Clifford translations and globally hyperbolic coverings}\label{sec3}

In the previous section we derived existence results for closed timelike geodesics which arose directly from our understanding and controlling certain variational aspects on the space of timelike geodesic loops introduced in section \ref{sec1.1}. In this section, we return to the specific context of Lorentzian manifolds which admit a covering by a globally hyperbolic spacetime, but now under the additional requirement that the latter possess certain symmetries, the so-called {\it Clifford translations}. On the one hand, global hyperbolicity simplifies some technical arguments, since the space of {\it causal curves} is already quite tame for globally hyperbolic spacetimes, and the need to control $TG$-homotopy classes is obviated. Therefore, no analysis (say) of the precompactness of these classes such as that used in section \ref{sec2} will be made here. But on the other hand, our general approach has inspired new proofs of existence results for closed timelike geodesics in this ambient which are slightly stronger than similar ones in the literature, and that has motivated us to revisit them here. The goal of this section is thus to make contact with some of these results. (See esp. \cite[Theorem 3.5]{guediri03}.)



The results in this section will be confined to {\it spacetimes}, as the conditions we will present are most natural in this context; accordingly, $(M,g)$ will be taken to be a spacetime for the remainder of the paper. Again, we briefly recall some basic definitions just to set notation and avoid ambiguities, referring to \cite{bookbeem96,bookoneill83} for details. 

On a spacetime $(M,g)$,  causal/timelike curves are divided in two exclusive classes: they can be either {\it future-directed}, if the tangent vectors along the curve lie all in the future causal cones, or {\it past-directed}, if the tangent vectors lie in the past cones. Given any set $A\subset M$, its {\it chronological future} is the set $I^+(A)$ of all points $p\in M$ which can reached via a future-directed timelike curve starting at $A$. The {\it causal future} of $A$ is the set $J^{+}(A)$ consisting of all points in $A$ itself, together with all those $p\in M$ which can be reached via a future-directed causal curve starting at $A$. The {\it chronological past} $I^-(A)$ (resp. {\it causal past} $J^-(A)$) can be defined in a dual fashion by taking past-directed timelike/causal curves in the above definition. It is well-known that $I^{\pm}(A)$ is always open, and for any $p \in M$ we denote $I^{\pm}(p):= I^{\pm}(\{p\})$ (resp. $J^{\pm}(p):= J^{\pm}(\{p\})$).

The {\it Lorentzian distance function} $d:= d_g:M\times M \rightarrow [0,+\infty]$  (alternatively called {\it time-separation function} in the literature) is given by 
\begin{equation}
    d(p,q) = \left\{ \begin{array}{cc}
         \sup _{\alpha \in \mathfrak{C}^+(p,q)} {\rm length}(\alpha)  ,&  \mbox{if $\mathfrak{C}^{+}(p,q) \neq \emptyset$},\\
         0, & \mbox{if $\mathfrak{C}^{+}(p,q) = \emptyset$},
    \end{array}\right. \nonumber
\end{equation}
where $\mathfrak{C}^+(p,q)$ denotes the set of all future-directed causal curves starting at $p$ and ending at $q$. In spite of its name, the Lorentzian distance function $d$ is {\it not} a distance in the sense of metric spaces (one reason why some authors prefer to call it {\it time-separation function} in the literature), and in particular we can well have $d(p,p) = +\infty$. It is always lower semicontinuous, but it may not be continuous. But it is finite-value and continuous when $(M,g)$ is globally hyperbolic. 

An isometry $\rho:M\rightarrow M$ of the spacetime $(M,g)$ is a \textit{Clifford translation} if $d(p,\rho(p))$ is constant for all $p\in M$. We say that the isometry $\rho$ is \textit{future timelike} if for each $p\in M$ either $\rho(p)=p$ or else $\rho(p)\in I^+(p)$. These two are in principle entirely different notions. For example, if one is given either a null or a spacelike complete Killing vector field $X$ possessing achronal orbits, then any stage $\rho_t$ ($t\in \mathbb{R}$) of its flow is a Clifford translation: $d(p, \rho_t(p)) \equiv 0$ for any $p\in M$. Even if the Killing field $X$ is everywhere future-directed timelike, if $(M,g)$ is {\it chronological}, i.e., one without timelike loops, then again $d(p, \rho_t(p)) \equiv 0$ for any $p\in M$ provided $t<0$, for then $\rho_t(p)$ is in the past of $p$. None of these examples yields a future timelike isometry except for $\rho_0 = Id_M$, which is trivially a future timelike Clifford translation in chronological spacetimes (since $d(p,p)\equiv 0$ for any $p\in M$ in this case\footnote{Curiously, by this definition, if $(M,g)$ is {\it totally vicious}, i.e, if through {\it any} point of $M$ there passes a timelike loop, then the identity map is again a Clifford translation, this time because $d(p,p) = +\infty$ for any $p\in M$.}). Such trivial Clifford translations/future timelike isometries will be consistently eschewed in what follows.

The following result due to Beem, Ehrlich and Markvorsen provides a relation between future timelike isometries and Clifford translations in a globally hiperbolic, future $1$-connected spacetime. (Recall also that $(M,g)$ is said to be \textit{future $1$-connected} if any two future-directed timelike curves connecting two given points are $t$-homotopic.)
\begin{thm}[\cite{beem&ehrlich&markvorsen88}, Theor. 4.2]\label{Teo_Beem_Isom_Clifford_Translacao}
Let $(M,g)$ be a globally hyperbolic, future $1$-connected spacetime with all timelike sectional curvatures $\geq 0$ and which is either timelike, or spacelike, or null geodesically complete. Then, any non-trivial future timelike isometry of $(M,g)$ is a Clifford translation.
\end{thm}

These considerations motivate a convenient definition.
\begin{defi}
An isometry $\rho$ of the spacetime $(M,g)$ is a \textit{future timelike Clifford translation} if $\rho$ is a future timelike isometry {\it and} a Clifford translation.
\end{defi}

It is easy to come up with simple examples to illustrate situations where the hypotheses in Thm. \ref{Teo_Beem_Isom_Clifford_Translacao} may fail and yet one still might have a future timelike Clifford translation.
\begin{exe}
\textcolor{white}{quebra de linha}
\begin{enumerate}
\item[1.]  Consider the Minkowski spacetime $\mathbb{R}^n_1$ with standard time orientation. The translation $\rho_a:\mathbb{R}^n_1\rightarrow \mathbb{R}^n_1$ given by $\rho_{a}(x):=x+a$, where $a \in \mathbb{R}^n$ is any future-directed timelike vector, is a non-trivial future timelike Clifford translation (note that $d(x,\rho_{a}(x))=|a|$ for every $x\in\mathbb{R}^n$).
\item[2.] More generally, let $(N^{n-1}, h)$ be any connected Riemannian manifold, and consider the Lorentzian cylinder $(M = \mathbb{R}\times N, -dt^2 \oplus h)$. It is a globally hyperbolic, geodesically complete spacetime if and only if $(N,h)$ is complete (cf. \cite[Thm. 3.67, p. 103]{bookbeem96}; the time orientation is chosen so that $\partial/\partial t $ is future-directed). It does not need to be timelike $1$-connected, however; indeed, it might even have infinitely many $t$-homotopy  classes, even if $N$ is simply connected \cite{sanchez2}, so Thm. \ref{Teo_Beem_Isom_Clifford_Translacao} cannot be applied directly. But for each positive number $a>0$ the map $\rho_a: (t,x) \in M \mapsto (t+a,x) \in M$ is evidently a non-trivial future timelike Clifford translation, with $d(p,\rho_a(p)) = a$ for any $p\in M$. 
\end{enumerate}
\end{exe}

We now give sufficient conditions for the existence of a locally maximizing timelike geodesic loop (see also \cite[Cor. 5.6]{costaesilva&flores&honorato21}). Note that we need neither compactness of $M$ nor regularity of the covering. 
\begin{thm}\label{prop_loc_max_loop}
Let $(M,g)$ be a Lorentz manifold which admits a Lorentzian covering $\phi:(\widetilde{M},\widetilde{g})\rightarrow (M,g)$ where $(\widetilde{M},\widetilde{g})$ is a globally hyperbolic spacetime. If the group $\mathcal{D}(\phi)\subset Diff(\widetilde{M})$ of deck transformations contains a non-trivial future timelike Clifford translation of $(\widetilde{M},\widetilde{g})$, then $(M,g)$ has a locally maximizing timelike geodesic loop at each point $p\in M$.
\end{thm}
\begin{proof}
First note that $M=\phi(\widetilde{M})$ is connected, because $\widetilde{M}$ is connected and $\phi$ is a continuous map. Let $p\in M$ and let $\rho:\widetilde{M}\rightarrow \widetilde{M}$ be a non-trivial future timelike Clifford translation in $\mathcal{D}(\phi)$. Fix some $\tilde{p}_1\in \phi^{-1}(\{p\})$. We have that $\tilde{p}_2:=\rho(\tilde{p}_1)\neq \tilde{p}_1$, because $\rho$ is a non-trivial deck transformation of $\phi$, and as $\rho$ is a future timelike isometry it follows that $\tilde{p}_2\in I^+(\tilde{p}_1)$. Since $(\widetilde{M},\widetilde{g})$ is globally hyperbolic then exists a maximal future timelike geodesic segment $\tilde{\sigma}:[0,1]\rightarrow \widetilde{M}$ that connects $\tilde{p}_1$ to $\tilde{p}_2$, i.e., $d(\tilde{p}_1,\rho(\tilde{p}_1))=d(\tilde{p}_1,\tilde{p}_2)=\comp(\tilde{\sigma})$. Let $U \ni p$ be a connected, evenly covered neighborhood of $p$ in $M$ and write $\phi^{-1}(U)= \bigsqcup _{i\in I\subset \mathbb{N}}\widetilde{U}_i$ (i.e., $I$ is at most countably infinite). Without loss of generality we can assume that $\tilde{p}_i\in\widetilde{U}_i$ for $i=1,2$. As consequence, we have $\rho(\widetilde{U}_1)=\widetilde{U}_2$. We denote by $\tilde{u}$ the initial velocity of $\tilde{\sigma}$. 

Let $\tilde{E}:\widetilde{\mathcal{D}}\subset T\widetilde{M}\rightarrow \widetilde{M}\times\widetilde{M}$ the map given by $\tilde{E}(v)=(\widetilde{\pi}(v),\widetilde{\exp}(v))$ for all $v\in\widetilde{\mathcal{D}}$, where $\widetilde{\mathcal{D}}$ is maximal domain of the exponential map $\widetilde{\exp}$ on $(\widetilde{M},\widetilde{g})$ and $\widetilde{\pi}$ is the canonical projection of $T\widetilde{M}$ onto $\widetilde{M}$. As $\tilde{u}\in\widetilde{\mathcal{D}}$ we have that

$$\tilde{E}(\tilde{u})=(\widetilde{\pi}(\tilde{u}),\widetilde{\exp}(\tilde{u}))= (\tilde{\sigma}(0),\tilde{\sigma}(1))=(\tilde{p}_1,\tilde{p}_2)\in\widetilde{U}_1\times\widetilde{U}_2.$$

Thus, $\widetilde{W}:=\tilde{E}^{-1}(\widetilde{U}_1\times\widetilde{U}_2)$ is an open neighborhood of $\tilde{u}$ in $\widetilde{\mathcal{D}}$. Now, since the map $\phi:(\widetilde{M},\widetilde{g})\rightarrow (M,g)$ is a Lorentzian covering then the derivative $d\phi:T\widetilde{M}\rightarrow TM$ of $\phi$ is a local diffeomorphism that preserves the causal character of vectors. In particular, the restriction of the bundle map $d\phi$ between the open subsets $T\widetilde{U}_1:=\sqcup_{\tilde{q}\in\widetilde{U}_1}T_{\tilde{q}}\widetilde{M}$ and $TU:=\sqcup_{q\in U}T_qM$ is a diffeomorphism. Note that $\widetilde{W}\subset T\widetilde{U}_1\cap\widetilde{\mathcal{D}}$. Let $\tilde{v}\in \widetilde{W}$. Consider the inextendible geodesic $\tilde{\sigma}_{\tilde{v}}$ in $\widetilde{M}$ such that $\dot{\tilde{\sigma}}_{\tilde{v}}(0)=\tilde{v}$. As $\tilde{v}\in\widetilde{\mathcal{D}}$ then $\tilde{\sigma}_{\tilde{v}}$ is defined on the interval $[0,1]$. Thus, the inextendible geodesic $\sigma_{v}:=\phi\circ\tilde{\sigma}_{\tilde{v}}$ in $M$ is such that $\dot{\sigma}_{v}(0)=v:=d\phi(\tilde{v})$ and it is also defined on the interval $[0,1]$. We conclude that the neighborhood $W:=(d\phi)(\widetilde{W})$ of $u:=(d\phi)(\tilde{u})$ is contained in $TU\cap\mathcal{D}$. 

We assert that the timelike geodesic loop $\sigma:=\phi\circ\tilde{\sigma}:[0,1]\rightarrow M$ with $\dot{\sigma}(0)=u\in\mathcal{D}$ has maximal length among all timelike geodesic loops with initial conditions in $W$. Indeed, let $\beta:[0,1]\rightarrow M$ be a timelike geodesic loop such that $\beta(0)=\beta(1)=x\in U$ and $w=\dot{\beta}(0)\in W$. Consider $\tilde{w}\in\widetilde{W}\subset \widetilde{\mathcal{D}}$ such that $d\phi(\tilde{w})=w$, and $\tilde{x}=\widetilde{\pi}(\tilde{w})\in \widetilde{U}_1$. Now consider the timelike geodesic segment $\tilde{\beta}:[0,1]\rightarrow \widetilde{M}$ such that $\tilde{\beta}(0)=\tilde{x}$ and $\dot{\tilde{\beta}}(0)=\tilde{w}$, i.e., $\tilde{\beta}(t)=\widetilde{\exp}_{\tilde{x}}\,(t\cdot \tilde{w})$ for all $t\in [0,1]$ (observe that $\tilde{\beta}(1)=\widetilde{\exp}_{\tilde{x}}\tilde{w} \in \widetilde{U}_2$ by the definition of $\widetilde{W}$). Note that $\beta=\phi\circ\tilde{\beta}$, since both geodesic segments $\beta$ and $\phi\circ \tilde{\beta}$ have the same initial conditions. Note besides that $\tilde{\beta}(1)=\rho(\tilde{x})$, because both points are in $\widetilde{V}_2$ and $\phi(\tilde{\beta}(1))=\phi(\rho(\tilde{x}))=x$. Since $\rho$ is a Clifford translation, then
\begin{eqnarray*}
\comp(\beta)&=&\comp(\tilde{\beta})\\
&\leq & d(\tilde{x},\rho(\tilde{x}))\\
& = & d(\tilde{p}_1,\rho(\tilde{p}_1))\\
& = & \comp(\tilde{\sigma}) = \comp(\sigma).
\end{eqnarray*}
\end{proof}



Theorem \ref{prop_loc_max_loop} together with Corollary \ref{cor1hillclimb} immediately yield a generalization of \cite[Theorem 3.5]{guediri03}:
\begin{cor}\label{teo_generaliz_resultado_guediri}
 Let $(M,g)$ be a Lorenztian manifold such that $Conj_c(p) =\emptyset$ for some $p \in M$. Assume $(M,g)$ admits a globally hyperbolic Lorentzian covering $\rho:(\widetilde{M},\widetilde{g})\rightarrow (M,g)$ such that the deck transformation group $\mathcal{D}(\rho)$ has a non-trivial future timelike Clifford translation, then $(M,g)$ has a closed timelike geodesic through $p$. In particular, if $(M,g)$ has non-negative sectional curvatures on timelike planes (cf. Remark \ref{remflaherty}), then  $(M,g)$ has a closed timelike geodesic passing through each one of its points.
\end{cor}
\qcd

Finally, the following corollary is again immediate from Theorem \ref{Teo_Beem_Isom_Clifford_Translacao} and Corollary \ref{teo_generaliz_resultado_guediri} recalling that any flat Lorentz space form is universally covered by Minkowski spacetime. It generalizes \cite[Corollary 3.6]{guediri03}.
\begin{cor}\label{Geod_Flat_Space_Form}
Let $(M,g)$ be a flat Lorentz space form. If the fundamental group $\pi_1(M)$ contains a non-trivial future timelike isometry, then $(M,g)$ contains a closed timelike geodesic through each one of its points.
\end{cor}
\qcd

Guediri showed that a compact flat Lorentz space form contains a closed timelike geodesic if and only if the fundamental group of the space form has a non-trivial future timelike isometry \cite[Theo. 4.1]{guediri03}. In view of Corollary \ref{Geod_Flat_Space_Form}, it is natural to wonder if the converse of the latter is also true. Therefore a natural question for future investigation is, assuming that $(M,g)$ is a flat Lorentz space form, whether it is true that the fundamental group $\pi_1(M)$ contains a non-trivial future timelike isometry if and only if $(M,g)$ contains a closed timelike geodesic.

\section*{Acknowledgments}
The authors are partially supported by the grant number PID2020-118452GBI00
(Spanish MICINN). JLF is also partially supported by A-FQM-494-UGR18 (FEDER, Andaluc\'{i}a), and KPRH by a scholarship funded
by Brazilian agency FAPESC (chamada p\'ublica no. 03/2017).

\end{document}